\documentclass[12pt]{article}
\usepackage{amsmath,amsfonts,amssymb,amsthm}
\usepackage{txfonts}%
\usepackage{bm}
\usepackage{url}

\usepackage{mathtools}
\usepackage{comment}
\global\long\def\T#1{#1^{\top}}
\mathtoolsset{showonlyrefs=true}
\def\ema{\end{pmatrix}}
\def\bma{\begin{pmatrix}}

\def\Re{{\,\rm Re\,}}

\def\rank{{\,\rm rank\,}}
\def\Ker{{\,\rm Ker\,}}

\def\mindeg{{\,\rm mindeg\,}}
\def\maxdeg{{\,\rm maxdeg\,}}
\def\ini{{\,\rm in\,}}
\def\p{\partial}

\theoremstyle{plain}
\newtheorem{theorem}{Theorem}[section]
\newtheorem{lemma}[theorem]{Lemma}
\newtheorem{corollary}[theorem]{Corollary}
\newtheorem{proposition}[theorem]{Proposition}
\theoremstyle{definition}
\newtheorem{definition}[theorem]{Definition}
\newtheorem{example}[theorem]{Example}
\theoremstyle{remark}
\newtheorem{remark}[theorem]{Remark}

\bmdefine{\Bz}{z}
\bmdefine{\Bx}{x}
\bmdefine{\Bt}{t}

\begin{document}

\title{Properties of powers of functions satisfying second-order linear differential equations with applications to statistics}

\author{
Naoki Marumo\thanks{Graduate School of Information Science and Technology, University of Tokyo}, \ 
Toshinori Oaku\thanks{Department of Mathematics, Tokyo Woman's Christian University} \ 
and Akimichi Takemura\footnotemark[1]\ 
}
\date{May, 2014}
\maketitle

\begin{abstract}
We derive properties of powers of a function satisfying a second-order linear differential equation.
In particular we prove that the $n$-th power of the function satisfies an $(n+1)$-th order differential equation and give a simple method for obtaining the differential equation. 
Also we determine the exponents of the differential equation and derive a bound for the degree of the polynomials, which are coefficients in the differential equation.
The bound corresponds to the order of differential equation satisfied by the $n$-fold convolution of the Fourier transform of the function.
These results are applied to some probability density functions used in statistics.
\end{abstract}

\noindent
{\it Keywords and phrases:} \ 
characteristic function,
exponents,
holonomic function,
indicial equation,
skewness

\section{Introduction}
\label{sec:intro}

In statistics it is important to study the distribution of a sum (i.e.\ convolution) of $n$ independent random variables.
Usually the distribution is studied through the characteristic function, because the convolution of probability density functions corresponds to the product of characteristic functions.
If the random variables are identically distributed, then we study the $n$-th power of a characteristic function.
The central limit theorem is proved by analyzing the limiting behavior of the $n$-th power of a characteristic function as $n\rightarrow\infty$.
Often the technique of asymptotic expansion is employed to improve the approximation for large $n$.
However for finite $n$, the exact distribution of the sum of random variables is often difficult to treat.
Hence it is important to develop methodology for studying properties of the $n$-th power of a function.

Recently techniques based on holonomic functions (\cite{kauers_book_11}, Chapter 6 of \cite{hibi_book_13}) have been introduced to statistics and successfully applied to some difficult distributional problems (e.g.\ \cite{nakayama}, \cite{hashiguchi}).
In this paper we investigate the case that the function satisfies a second-order linear differential equation with rational function coefficients, which we call {\em holonomic differential equation}.
In Section 2 we prove that the $n$-th power satisfies an $(n+1)$-th order differential equation and give a simple method for obtaining the differential equation. 
Also we determine the exponents of the differential equation and derive a bound for the degree of the polynomials which appear as coefficients of the differential equation.

As shown in Section 3, there are some important examples in statistics which falls into this case.
We discuss sum of beta random variables and sum of cubes of standard normal random variables.
The differential equations reveal many interesting
properties of the characteristic function and the probability density function of the sum of random variables.
These properties are hard to obtain by other methods.
We end the paper with some discussions in Section 4.

\section{Main results}
\label{sec:main}
In this section we present our main results in Theorems \ref{thm:ode}, \ref{thm:deg} and
\ref{thm:exponents}.  Theorem \ref{thm:ode} gives the differential equation satisfied by the $n$-th power.
Theorem \ref{thm:deg} bounds the degree of coefficient polynomials.
Theorem \ref{thm:exponents} derives exponents of the differential equation.

Let $\mathbf C(x)$ denote the field of rational functions in $x$ with complex coefficients and let
\[
R=\mathbf C(x) \langle \p_x \rangle,  \quad \partial_x = \frac{d}{dx},
\]
denote the ring of differential operators with rational function coefficients.
In $R$, 
the product of $\p_x$ and $a(x)\in \mathbf C(x)$ is defined as $\p_x a(x) = a(x) \p_x + a'(x)$, where $a'(x)$ is the derivative of $a(x)$ with respect to $x$.
In order to distinguish the product  in $R$  and the action of $\p_x$ to a function, we denote the latter by the symbol $\bullet$.
\begin{example}
If we write $\p_x x$, both $\p_x$ and $x$ are the elements of $R$.
Hence $\p_x x = x\p_x+1$.
On the other hand, if we write $\p_x \bullet x$, this $x$ is a function.
Hence $\p_x \bullet x = 1$.
\end{example}

In this paper we study $f(x)$ which is a holonomic function satisfying a second-order differential equation:
\begin{equation}
\Big[ \p_x^2 - a_1(x)\p_x - a_0(x) \Big] \bullet f(x)
= 0,  \qquad a_0(x), a_1(x)\in \mathbf C(x). \label{ode 1}
\end{equation}

\subsection{Order of the differential equation of the $n$-th power and its Fourier transform}
\label{subsec:order-n-th-power}

Let $\bm q_0 = \T{(1,\ 0,\ \dots,\ 0)}$ be an $(n+1)$ dimensional column vector and let
\begin{equation}
A(x)=
\bma
0 & a_0(x) \\
n &  a_1(x) & 2a_0(x) \\
& n-1 & 2a_1(x) & \ddots \\
&& \ddots & \ddots & na_0(x)\\
&&& 1 & na_1(x)
\ema 
\label{eq:qk0}
\end{equation}
be an $(n+1)\times (n+1)$ tridiagonal matrix with entries from $\mathbf C(x)$.
Furthermore define
\begin{align}
\tilde A (x,\p_x) &= A(x) + \p_x I  \label{eq:q_k}
\\
&=\bma
\p_x  & a_0(x) \\
n & \p_x + a_1(x) & 2a_0(x) \\
& n-1 & \p_x + 2a_1(x) & \ddots \\
&& \ddots & \ddots &  (n-1)a_0(x) \\
&&& 2 & \p_x + (n-1)a_1(x) & na_0(x)\\
&&&& 1 & \p_x + na_1(x)
\ema
\nonumber
\end{align}
with entries from $R$.
Let
\begin{align}
Q(x)=(q_{ij}(x))_{\substack{0\leq i \leq n\\0 \leq j \leq n+1}} = (\bm q_0 ,\ \tilde A(x,\p_x) \bullet \bm q_0 ,\  \tilde A(x,\p_x)^2 \bullet \bm q_0 ,\ \dots,\ \tilde A(x,\p_x)^{n+1} \bullet \bm q_0)
\end{align}
be an $(n+1)\times (n+2)$ matrix with entries from $\mathbf C(x)$. 
If we write ${\bm q}_j = \tilde A(x,\p_x)^j \bullet {\bm q}_0$, $j=0,\dots,n+1$, then
\[
{\bm q}_{j+1}=\tilde A(x,\p_x) \bullet {\bm q}_j, 
\]
or writing down the elements we have
\begin{equation}
\label{eq:qij}
q_{i,j+1}(x)
= (n+1-i)q_{i-1,j}(x) + (\p_x + ia_1(x)) \bullet q_{i,j}(x) + (i+1)a_0(x) q_{i+1,j}(x), 
\end{equation}
where $q_{-1,j}(x)=q_{n+1,j}(x)=0$.
Hence it is easy to compute the elements of the columns of $Q(x)$ recursively, starting from the first column.

Define 
\begin{equation}
\label{eq:falling-factorial}
[n]_i= \prod_{k=0}^{i-1} (n-k), \qquad ([n]_0=1).
\end{equation}
From \eqref{eq:qij} we can easily prove 
that $Q(x)$ is an upper-triangular matrix with non-zero
diagonal elements, although $Q(x)$ is not a square matrix (cf.\ Example \ref{ex:qx} below).

\begin{lemma}\label{lem:tri}
$q_{ij}(x)=0$ if $i>j$.
$q_{ii}(x) = [n]_i \neq 0 \ \ (i=0,1,\dots,n)$.
\end{lemma}
\begin{proof}
We use induction on $j$.
The result is trivial for $j=0$.
Assume $q_{ij}(x) = 0\ (i>j)$ and $q_{j,j}(x) = [n]_j$.
Then by \eqref{eq:qij} we have
$q_{i,j+1}(x) = 0\ (i>j+1)$ and $q_{j+1,j+1}(x) = (n-j)[n]_j=[n]_{j+1}$.
\end{proof}

This lemma implies $\rank Q(x)=n+1$, or $\dim \Ker Q(x) = 1$.
Hence the element of $\Ker Q(x)$ is unique up to the multiplication of a rational function.
Here note that we are using the linear algebra over $\mathbf C(x)$.

Let 
\begin{equation}
\label{eq:kernel-vector}
\bm v(x)=(v_{i}(x))_{0\leq i \leq n+1}  \in \Ker Q(x), \quad \bm v(x)\neq \bm 0,
\end{equation}
where 
$v_i(x)\in \mathbf C(x)$, $i=0,\dots,n+1$.
Once we set $v_{n+1}(x)\neq 0$, then by the triangularity of $Q(x)$, 
$v_{n}(x), v_{n-1}(x), \dots, v_0(x)$ are successively determined.
Moreover, if we set $v_{n+1}(x)=0$, then we obtain  $v_n(x)=\dots=v_0(x)=0$.
Hence $v_{n+1}(x)\neq 0$ for $\bm v(x)\neq \bm 0$. 
Often we set $v_{n+1}(x)=1$.  For theoretical investigation it is convenient to clear the common
denominators of $v_i(x)$'s and take $v_i(x)$'s as polynomials.

\begin{example}
\label{ex:qx}
Let $n=3$ and let $a_0(x) = 1+x^{-2},\ a_1(x) = -x^{-1}$.
Then
\begin{align}
Q(x)=
\bma
 1 &  0 &  3 + 3x^{-2} &  -3x^{-1} - 9x^{-3} & 21 + 51x^{-2} + 54x^{-4} \\
 0 &  3 &       -3x^{-1} &    21 + 27x^{-2} & -66x^{-1} - 144x^{-3}\\
 0 &  0 &          6 &               -18x^{-1} &          60 + 126x^{-2} \\
 0 &  0 &          0 &                   6 &                   -36x^{-1} \\
\ema
\end{align}
If we set $v_4(x)=1$, we successively obtain
\begin{align}
v_3(x) = 6x^{-1},\quad
v_2(x) = -10-3x^{-2},\quad
v_1(x) = -30x^{-1}-9x^{-3},\quad
v_0(x) = 9 + 6x^{-2} + 9x^{-4}.
\end{align}
Multiplying by $x^4$ we obtain $\bm v$ with polynomial elements.
\end{example}

We now derive a holonomic differential equation satisfied by the $n$-th power of the holonomic function $f(x)$.

\begin{theorem}\label{thm:ode}
The $n$-th power of $f(x)$ satisfies the following $(n+1)$-th order holonomic differential equation:
\begin{align}
\Big[
v_{n+1}(x)\p_x^{n+1} + v_n(x) \p_x^n + \dots + v_1(x) \p_x + v_0(x)
\Big]\bullet  f(x)^n
=0, \label{ode n}
\end{align}
where $v_i(x)$'s are given in \eqref{eq:kernel-vector}.
\end{theorem}
\begin{proof}
By induction we prove
\begin{align}
\p_x^k \bullet f^n
= (f^n ,\ f^{n-1} f',\ \dots,\ f f'^{n-1},\ f'^n)\ \tilde A(x,\p_x)^k \bullet \bm q_0 \label{ind 1}
\end{align}
for any $k\ge 0$. 
It is obvious for $k=0$.
By \eqref{ode 1},
\begin{align}
&\quad \p_x \bullet \Big(f(x)^{n-j}f'(x)^j\Big)\\
& = j f(x)^{n-j} f'(x)^{j-1}f''(x) + (n-j)f(x)^{n-j-1} f'(x)^{j+1}\\
& = j a_0(x)f(x)^{n-j+1}f'(x)^{j-1} + j a_1(x) f(x)^{n-j} f'(x)^j + (n-j)f(x)^{n-j-1} f'(x)^{j+1}
\end{align}
holds for all $j=0,1,\dots,n$, and this leads to
\begin{equation}
\p_x \bullet
(f^n ,\ f^{n-1} f',\ \dots,\ f f'^{n-1},\ f'^n)
= (f^n ,\ f^{n-1} f',\ \dots,\ f f'^{n-1},\ f'^n)\ A(x),
\label{eq:ax1}
\end{equation}
where $A(x)$ is given in \eqref{eq:qk0}.

Hence, assuming \eqref{ind 1} for $k$, we obtain
\begin{align*}
\p_x^{k+1} \bullet f^n
&= \Big( \p_x \bullet (f^n ,\ f^{n-1} f',\ \dots,\ f f'^{n-1},\ f'^n) \Big)\ \tilde A(x,\p_x)^k \bullet \bm q_0\\
&\qquad + (f^n ,\ f^{n-1} f',\ \dots,\ f f'^{n-1},\ f'^n)\ \p_x \bullet \Big( \tilde A(x,\p_x)^k \bullet \bm q_0 \Big)\\
&=  (f^n ,\ f^{n-1} f',\ \dots,\ f f'^{n-1},\ f'^n)\,A(x) \ \tilde A(x,\p_x)^k \bullet \bm q_0  \qquad\qquad (\text{by \eqref{eq:ax1}})\\
&\qquad + (f^n ,\ f^{n-1} f',\ \dots,\ f f'^{n-1},\ f'^n)\ \p_x \tilde A(x,\p_x)^k \bullet \bm q_0\\
&= (f^n ,\ f^{n-1} f',\ \dots,\ f f'^{n-1},\ f'^n)\ (A(x) +\p_x I)\ \tilde A(x,\p_x)^k \bullet \bm q_0 \\
&= (f^n ,\ f^{n-1} f',\ \dots,\ f f'^{n-1},\ f'^n)\ \tilde A(x,\p_x)^{k+1} \bullet \bm q_0.  
\qquad\qquad (\text{by \eqref{eq:q_k}})
\end{align*}
Thus \eqref{ind 1} is proved.
By arranging \eqref{ind 1} for $k=0,1,\dots,n+1$, we have
\begin{align}
(f^n,\ \p_x\bullet f^n,\ \dots,\ \p_x^{n+1}\bullet f^n) = (f^n ,\ f^{n-1} f',\ \dots,\ f f'^{n-1},\ f'^n)\ Q(x).
\end{align}
By multiplying it by $\bm v(x) \in \Ker Q(x)$ from the right, we obtain \eqref{ode n}.
\end{proof}

\begin{remark}
If we just want to show the existence of a holonomic differential equation of order $n+1$,
we have only to consider 
\begin{align}
M = \mathbf C(x) f^n + \mathbf C(x) f^{n-1}f' + \cdots +  \mathbf C(x) ff'^{n-1} + \mathbf C(x) f'^n.
\end{align}
Then $M$ is a left $R$-module as well as a vector space over $\mathbf C(x)$ of dimension at most $n+1$.
Hence $n+2$ elements, $f^n$, $\p_x\bullet f^n$, \dots, $\p_x^{n+1}\bullet f^n$, 
which belong to $M$, are linearly dependent  over $\mathbf C(x)$.
Similarly, we see that when $f(x)$ satisfies a holonomic differential equation of order $r$ $(\geq 3)$, $f(x)^n$ satisfies a holonomic differential equation of order $\binom{n+r-1}{r-1}$.
\end{remark}


There exists a function $f(x)$ satisfying a second-order holonomic differential equation, 
such that $f(x)^n$ does not satisfy any 
holonomic differential equation of order less than $n+1$.
\begin{example}
Let $f(x)=\sin x$, with $f''(x)+f(x)=0$.
We prove by contradiction that $f^n,\ f^{n-1}f', \dots, f'^n$, or $\sin^n x$, $\sin^{n-1}x \cos x$, $\dots$, $\cos^n x$ are linearly independent over $\mathbf C(x)$. 
It is obvious for $n=0$.
Let $m \geq 1$ be the smallest integer such that $\sin^m x, \dots, \cos^m x$ are linearly dependent.
Then, there exist rational functions $q_0(x),\dots,q_m(x)$, not all zero, such that
\begin{align}
q_0(x)\sin^m x + q_1(x)\sin^{m-1}x \cos x+\dots+ q_{m-1}(x)\sin x \cos x^{m-1}+q_m(x)\cos^m x=0. \label{l i}
\end{align}
By putting $x=k\pi\ (k=0,1,\dots)$, $q_m(x)$ has infinite number of zeros, and therefore $q_m(x)$ is identically zero.
Divide the equation \eqref{l i} by $\sin x$, and we obtain $q_0(x)\sin^{m-1} x +\dots+q_{m-1}(x)\cos^{m-1} x=0$, which is a contradiction.

Since $f^n,\ f^{n-1}f', \dots, f'^n$ are linearly independent and the matrix $\tilde Q(x) = (\bm q_0 ,\ \tilde A \bullet \bm q_0 ,\ \dots,\ \tilde A^n \bullet \bm q_0)$ is non-singular by Lemma \ref{lem:tri} over $\mathbf C(x)$,
$(f^n,\ \p_x \bullet f^n,\ \dots,\ \p_x^n \bullet f^n) = (f^n,\ f^{n-1}f', \dots, f'^n)\tilde Q(x)$ are linearly independent over $\mathbf C(x)$. 
Thus, there does not exist a holonomic differential equation of order less than $n+1$ satisfied by $f(x)^n=\sin^n x$.
\end{example}

We have already remarked that we can take $v_i(x)$, $0\le i\le n+1$, as polynomials in \eqref{ode n}. 
Also we can cancel common factors in them.
Hence we can assume that they are coprime polynomials.
We now investigate the highest degree of these  polynomials, 
which is important when the differential equation is Fourier transformed, 
because it is equal to the order of the transformed equation.

For the rest of this subsection we assume that $a_0(x),a_1(x)$ are Laurent polynomials.
Here, we define mindeg and maxdeg of a Laurent polynomial.
\begin{definition}
For a non-zero Laurent polynomial $f(x)=\sum_{k=m}^M c_k x^k\ (m<M,\ c_m\neq 0,\ c_M\neq 0)$, we define
\begin{align}
\mindeg f(x) = m,\qquad
\maxdeg f(x) = M.
\end{align}
We define $\mindeg 0 = \infty$, $\maxdeg 0 = -\infty$.
\end{definition}
Note that for a polynomial $f(x)$, $\maxdeg f(x)=\deg f(x)$.

Now we state the following theorem on the largest degree of the polynomials.

\begin{theorem}\label{thm:deg}
Assume that $a_0(x), a_1(x)$ in \eqref{ode 1} are Laurent polynomials and 
let $m_i = \mindeg a_i(x)$, $M_i = \maxdeg a_i(x)$, $i=0,1$.
Let $v_0(x),\ v_1(x),\ \dots,\ v_{n+1}(x)$ be coprime polynomials in \eqref{ode n}.
If $m_1 \leq -1,\ M_1\geq -1,\ m_0\geq 2m_1,\ M_0\leq 2M_1$, then
\begin{align}
\max_{0\leq k \leq n+1} \deg v_k(x)
\leq \max\{M_0+(n-1)M_1, nM_1, 0 \} - \min\{ m_0, m_1 \} - (n-1)m_1. \label{maxdeg}
\end{align}
\end{theorem}
\begin{proof}
Let $m_{ij}$ denote $\mindeg q_{ij}(x)$.
We prove
\begin{align}
m_{0j} \geq m_0 + (j-2)m_1,\qquad
m_{ij} \geq (j-i)m_1,\quad
(0<i<j),\label{degin}
\end{align}
for $j=2,3,\dots,n+1$ by induction.
It is easy to check them for $j=2$.
Assuming them up to $j$, by \eqref{eq:qij}, we have
\begin{align}
m_{0,j+1}
&\geq \min\{ m_{0j} + m_1,\  m_{1j} + m_0\}\\
&= m_0 + (j-1)m_1,\\
m_{i,j+1}
&\geq \min\{ m_{i-1,j},\ m_{ij} + m_1,\ m_{i+1,j} + m_0 \}\\
&\geq \min\{ m_{i-1,j},\ m_{ij} + m_1,\ m_{i+1,j} + 2m_1 \}\\
&= (j-i+1)m_1, \qquad(0<i<j+1).
\end{align}
Thus, the results are shown by induction.

Hence choosing an element $\tilde {\bm v}(x) = \T{(\tilde v_0(x),\ \dots,\ \tilde v_n(x),\ 1)}$ $\in \Ker Q(x)$, we successively obtain
\begin{align}
&\mindeg \tilde v_n(x) \geq m_1,\quad
\mindeg \tilde v_{n-1}(x) \geq 2m_1,\\
&\qquad \dots,\quad \mindeg \tilde v_1(x) \geq nm_1,\quad
\mindeg \tilde v_0(x) \geq m_0 + (n-1)m_1.
\end{align}
This implies that $\min_{k} \mindeg \tilde v_k(x) \geq \min\{ m_0+(n-1)m_1,\  nm_1,\ 0 \} = \min\{m_0,\ m_1\} + (n-1)m_1$.
By regarding as Laurent polynomials of $x$ as those of $x^{-1}$, we also have $\max_{k} \maxdeg \tilde v_k(x) \leq \max\{ M_0+(n-1)M_1,\  nM_1,\ 0 \}$.
Therefore, clearing the denominators of $\tilde {\bm v}$ of \eqref{ode n}, we obtain \eqref{maxdeg} for the polynomials $v_i(x)$ of $\bm v$.
\end{proof}

Let $D=\mathbf C\langle x, \p_x \rangle$ denote the polynomial ring  in $x$ and $\p_x$ with complex coefficients.
The Fourier transform $\mathcal F$, which is a ring isomorphism of $D$, is defined by (Section 6.10 of \cite{hibi_book_13})
\begin{align}
\mathcal F : x\mapsto i\p_x,\qquad
\mathcal F : \p_x\mapsto ix,\qquad
(i=\sqrt{-1}).
\label{eq:fourier-corespondence}
\end{align}
Hence the Fourier transform $\hat L(x,\p_x)$ of $L(x,\p_x)\in D$ is given by $L(i\p_x, ix)$.

This definition is based on the fact that 
if a function $f(x)$ satisfies the differential equation $L(x,\p_x) \bullet f(x) = 0$, then the Fourier transform $\hat f(\xi) = \int_{-\infty}^\infty e^{-ix\xi} f(x)\, dx$ satisfies the differential equation $\hat L(x,\p_x) \bullet \hat f(x)=0$ under some regularity conditions.
If $f$ is a rapidly decreasing holonomic function, then the correspondence 
\eqref{eq:fourier-corespondence} is immediate (Section 5.1.4 of \cite{stein_book}). 
The correspondence can be justified
in the class of slowly increasing functions.
See Chapter 5 of \cite{grubb_book}.

We take $v_i(x)$'s as coprime polynomials in \eqref{ode n} and then take the Fourier transform.
By the correspondence \eqref{eq:fourier-corespondence},
the highest degree of the coefficient polynomials of $L$ equals the order of $\hat L$.  
Hence we have the following corollary.

\begin{corollary}\label{cor:con}
Under the condition of Theorem \ref{thm:deg}, 
there exists a holonomic differential equation satisfied by the $n$-th convolution of $\mathcal F[f(x)]$ whose order is less than or equal to the right-hand side of 
\eqref{maxdeg}:
\[
\max\{M_0+(n-1)M_1,\  nM_1,\ 0 \} - \min\{ m_0,\ m_1 \} - (n-1)m_1.
\]
\end{corollary}


\subsection{Exponents for the differential equation of the $n$-th power and the Fourier transformed equation}
\label{subsec:exponents}

Consider an $r$-th order differential equation
\begin{equation}
\Big[  (x-a)^r \p_x^r +  (x-a)^{r-1} b_{r-1}(x)\p_x^{r-1} + \dots +  (x-a) b_1(x)\p_x + b_0(x)  \Big] \bullet f(x) = 0. 
\label{eq:r-th-equation}
\end{equation}
If $b_0(x),\dots,b_{r-1}(x)$ are all analytic at $x=a$, then $a$ is said to be a regular singular point for the equation.
If $a=0$ and $b_0(1/x),\dots,b_{r-1}(1/x)$ are all analytic at $x=0$, then $\infty$ is said to be a regular singular point for the equation.

When the equation \eqref{eq:r-th-equation} is holonomic, $a$ is a regular singular point if the denominators of $b_0(x),\dots,b_{r-1}(x)$ do not have a factor $(x-a)$,
and $\infty$ is a regular singular point if $a=0$ and $b_0(x),\dots,b_{r-1}(x)$ are all proper.

When $x_0 \in  \mathrm C \cup \{\infty\}$ is a regular singular point for the equation,
the $r$-th degree equation
\begin{align}
b(\lambda) = [\lambda]_r + b_{r-1}(x_0) [\lambda]_{r-1} + \dots + b_1(x_0)[\lambda]_1 + b_0(x_0) = 0,
\end{align}
where $[\lambda]_i=\lambda(\lambda-1)\dots(\lambda-i+1)$ (cf.\ \eqref{eq:falling-factorial}), 
is called the {\em indicial equation} (Section 9.5 of \cite{hille_book}, Chapter 15 of \cite{ince_book}) 
for \eqref{eq:r-th-equation}
relative to the regular singular point $x_0$.
The roots of the indicial equation are called the {\em exponents}.

The case $x_0\neq \infty$ can be reduced to the case $x_0=0$ by the transform $x-x_0 \mapsto x$
and the case $x_0 = \infty$ can be reduced to $x_0=0$ by $x \mapsto 1/x$.
Hence in the following we put $x_0=0$.

The equation \eqref{eq:r-th-equation} is equal to
\begin{equation}
\Big[ [\theta_x]_r +  b_{r-1}(x) [\theta_x]_{r-1} + \dots +  b_1(x)[\theta_x]_1 + b_0(x)  \Big] \bullet f(x) = 0,
\end{equation}
where $\theta_x=x\p_x$ is the Euler operator, since $x^k\p_x^k = [\theta_x]_k$.
This shows that $b(\lambda)$ is obtained by expressing the differential equation in terms of $x$ and $\theta_x$, and substituting $x=0$ and $\theta_x = \lambda$ formally.

In this subsection 
we assume that $x_0  \in \mathrm C \cup \{\infty\}$ is a regular singular point for the equation \eqref{ode 1} for $f(x)$.
Let $\lambda_1,\lambda_2$ be the exponents for \eqref{ode 1} relative to the regular singular point $x_0$.

We show the following lemma on the eigenvalues of a matrix before the proof of 
Theorem \ref{thm:exponents} on the exponents for \eqref{ode n} relative to $x_0$.

\begin{lemma}\label{lem:eig}
The eigenvalues of an $(n+1)\times(n+1)$ tridiagonal matrix
\begin{align}
M =
\bma
0  & -\lambda_1\lambda_2 \\
n & \lambda_1+\lambda_2 & -2\lambda_1\lambda_2 \\
& n-1 & 2(\lambda_1+\lambda_2) & \ddots \\
&& \ddots & \ddots &  -(n-1)\lambda_1\lambda_2 \\
&&& 2 & (n-1)(\lambda_1+\lambda_2) & -n\lambda_1\lambda_2\\
&&&& 1 & n(\lambda_1+\lambda_2)
\ema
\end{align}
are 
\[
(n-k) \lambda_1 + k \lambda_2, \qquad (k=0,1,\dots,n).
\]
\end{lemma}
\begin{proof}
The eigenvalues of $M$ are equal to those of the matrix
\begin{align}
&M' =
\bma
0  & \lambda_1 \\
-n\lambda_2 & \lambda_1+\lambda_2 & 2\lambda_1 \\
& -(n-1)\lambda_2 & 2(\lambda_1+\lambda_2) & \ddots \\
&& \ddots & \ddots &  (n-1)\lambda_1 \\
&&& -2\lambda_2 & (n-1)(\lambda_1+\lambda_2) & n\lambda_1\\
&&&& -\lambda_2 & n(\lambda_1+\lambda_2)
\ema
.
\end{align}
because the determinant of a tridiagonal matrix $T=(t_{ij})$ depends only on the 
diagonal elements $t_{ii}$ and the products of off-diagonal elements $t_{i,i+1}t_{i+1,i}$.

If $\lambda_1=0$, it is obvious that the eigenvalues are $0,\ \lambda_2,\ \dots,\ n\lambda_2$.
Otherwise, putting $z=\lambda_2/\lambda_1$, we prove that the eigenvalues of the matrix $M'/\lambda_1$ are $\mu_k = kz+(n-k)\ (k=0,1,\dots,n)$.

For $z\neq 1$, all of $\mu_k$'s are different.
We show that the eigenvector corresponding to $\mu_k$ is $\bm v^k =(v_l^k)_{0\leq l \leq n}$ where
\begin{align}
v_l^k
= \sum_{j} \binom{n-k}{l-j} \binom{k}{j} z^j. \label{v_l^k}
\end{align}
Here, the summation for $j$ is over the finite interval $\max\{0,\ k+l-n \} \leq j \leq \min\{ k,\ l \}$.

The $l$-th entry $(0\leq l \leq n)$ of $(\mu_k I-M'/\lambda_1)\bm v^k$ equals
\begin{align}
\sum_j \Bigg[ 
&(n-l+1)\binom{n-k}{l-j}\binom{k}{j-1}
+ (n-k-l)\binom{n-k}{l-j}\binom{k}{j}\\
&\qquad + (k-l) \binom{n-k}{l-j+1}\binom{k}{j-1}
- (l+1) \binom{n-k}{l-j+1}\binom{k}{j}\Bigg] z^j. \label{sumbin}
\end{align}
The first two terms equal
\begin{align}
&\quad \Bigg[ (k+1)\binom{k}{j-1} + (n-k-l)\binom{k+1}{j} \Bigg] \binom{n-k}{l-j} \\
& = (n-k-l+j)\binom{n-k}{l-j} \binom{k+1}{j}
=(n-k)\binom{n-k}{l-j+1}\binom{k+1}{j},
\end{align}
by the relations $\binom{n}{k-1}+\binom{n}{k}=\binom{n+1}{k}$, $k\binom{n}{k}=n\binom{n-1}{k-1}$ and $\binom{n}{k}=\binom{n}{n-k}$.
Similarly, the last two terms equal $-(n-k)\binom{n-k}{l-j+1}\binom{k+1}{j}$.
Those show that $(\mu_k I-M'/\lambda_1)\bm v^k = \bm 0$.

For $z=1$, all of $\mu_k$'s are identical.
Let $\bm v^k =(v_l^k)_{0\leq l \leq n}$ $(k=0,1,\dots,n)$ be $(n+1)$ dimensional vectors where
\begin{align}
v_l^k = \binom{n-k}{l}, \quad
(0\leq l \leq n-k),\qquad
v_l^k = 0,\quad
(n-k < l \leq n).
\end{align}
Then, we can show $(nI-M'/\lambda_1)\bm v^0=\bm 0$ and $(nI-M'/\lambda_1)\bm v^k = k \bm v^{k-1}$ $(k=1,\dots,n)$ as above.
Hence $\bm v^k$'s, which are linearly independent,  are the generalized eigenvectors of the matrix.
\end{proof}

\begin{remark}
\eqref{v_l^k} can be formally written as follows:
\begin{align}
v_l^k
= \binom{n-k}{l} \,{}_2F_1(-k,-l,n-k-l+1;z),
\end{align}
where
\begin{align}
{}_2F_1(a,b,c;z)
= \sum_{n=0}^\infty \frac{(a)_n (b)_n}{(c)_n} \frac{z^n}{n!},\qquad
(a)_n = \prod_{k=0}^{n-1} (a+k).
\end{align}
Then, the $l$-th entry $(0\leq l \leq n)$ of $(\mu_k I-M'/\lambda_1)\bm v_k = \bm 0$ is equivalent to
\begin{align}
& c(c-1) \, {}_2F_1(-k,-l-1,c-1;z)
 - c\big( c-1+(k-l)z \big) \, {}_2F_1(-k,-l,c;z)\notag\\
&\qquad\qquad\qquad 
-l(c+k) z\,{}_2F_1(-k,-l+1,c+1;z) = 0,
\label{eq:2f1-recursion}
\end{align}
where $c=n-k-l+1$.
The recursion \eqref{eq:2f1-recursion} can be confirmed by \texttt{HolonomicFunction} (\cite{hfs}), a package of {\tt Mathematica}.
\end{remark}

We now show the following theorem on the exponents for \eqref{ode n}.

\begin{theorem}
\label{thm:exponents} If $x_0$ is a regular singular point for \eqref{ode 1}, then
$x_0$ is a regular singular point for \eqref{ode n}.
Moreover, its exponents for \eqref{ode n} are
\begin{align}
(n-k) \lambda_1 + k \lambda_2, \qquad (k=0,1,\dots,n).
\end{align}
\end{theorem}
\begin{proof}
We put $x_0=0$ without loss of generality by translation.
Then, the equation \eqref{ode 1} can be rearranged to
\[
\Big[ \theta_x^2 - b_1(x)\theta_x - b_0(x) \Big] \bullet f(x)
= 0, 
\]
where $b_0(x) = x^2a_0(x)$ and $b_1(x) = xa_1(x)+1$ are analytic at $x=0$.
%

Let
\begin{equation}
B(x)=
\bma
0 & b_0(x) \\
n &  b_1(x) & 2b_0(x) \\
& n-1 & 2b_1(x) & \ddots \\
&& \ddots & \ddots & nb_0(x)\\
&&& 1 & nb_1(x)
\ema
\end{equation}
be an $(n+1)\times(n+1)$ tridiagonal matrix and let $\tilde B(x,\theta_x) = B(x)+\theta_x I$, similarly to \eqref{eq:qk0} and \eqref{eq:q_k}.
Then, as in \eqref{ind 1} we have
\begin{equation}
\theta_x^k \bullet f^n
= (f^n ,\ f^{n-1} (\theta_x \bullet f),\ \dots,\ f(\theta_x \bullet f)^{n-1}, \ (\theta_x \bullet f)^n)\ \tilde B(x,\theta_x)^k \bullet \bm q_0.
\label{theta k}
\end{equation}
Let
\begin{align}
P(x) = (p_{ij}(x))_{\substack{0\leq i \leq n\\0 \leq j \leq n+1}} = (\bm q_0,\ \tilde B(x,\theta_x) \bullet \bm q_0,\ \dots,\ \tilde B(x,\theta_x)^{n+1} \bullet \bm q_0)
\end{align}
be an $(n+1)\times(n+2)$ matrix and let $\bm w(x)=(w_0(x),\ \dots,\ w_n(x),\ 1)  \in \Ker P(x)$.
Then, the differential equation
\begin{align}
\Big[
\theta_x^{n+1} + w_n(x) \theta_x^n + \dots + w_1(x) \theta_x + w_0(x)
\Big]\bullet  f(x)^n
=0
\end{align}
is equal to \eqref{ode n}.

Every entry of $P(x)$ is analytic at $x=0$ since $b_0(x)$ and $b_1(x)$ are analytic.
Moreover $w_0(x),\dots,w_n(x)$ are all analytic because each of $p_{ii}(x)$ is a constant and $w_{n+1}(x)=1$.
Thus, $x_0=0$ is a regular singular point for \eqref{ode n} since $\theta_x^k$ is a linear combination of $1,x\p_x,\dots,x^k\p_x^k$.
Furthermore, the indicial equation for \eqref{ode n} is
\begin{align}
\lambda^{n+1} + w_n(0) \lambda^n + \dots + w_1(0) \lambda + w_0(0) = 0.
\end{align}

On the other hand, since
\begin{align}
\tilde B(x,\theta_x)^k \bullet \bm q_0 \Big|_{x=0}
&= \tilde B(x,\theta_x)^{k-1} \bullet B(0) \bm q_0 \Big|_{x=0}
= \cdots =  \tilde B(x,\theta_x) \bullet B(0)^{k-1} \bm q_0 \Big|_{x=0}
=  B(0)^k \bm q_0,
\end{align}
we have
\begin{align}
P(0)
= (\bm q_0,\ B(0)\bm q_0,\ \dots,\ B(0)^{n+1}\bm q_0).
\end{align}
Hence, by $P(0)\bm w(0)=\bm 0$, we obtain
\begin{align}
\Big( B(0)^{n+1} + w_n(0) B(0)^n + \dots + w_1(0) B(0) + w_0(0)I \Big) \bm q_0 = \bm 0.
\end{align}
Then by the Cayley--Hamilton theorem and by the uniqueness of $\bm w(0)$, the characteristic equation of the matrix $B(0)$ is equal to the indicial equation for \eqref{ode n}.

On the other hand, by Vieta's formula, we have $b_0(0) = -\lambda_1\lambda_2$ and $b_1(0) = \lambda_1 + \lambda_2$.
Hence the matrix $B(0)$ is equal to the matrix $M$ in Lemma \ref{lem:eig}.
Thus the exponents for \eqref{ode n} are proved to be $(n-k)\lambda_1+k\lambda_2$ $(k=0,1,\dots,n)$. 
\end{proof}


We have described the exponents for the differential equation satisfied by $f(x)^n$.
From now on, we investigate the exponents for the Fourier transformed equation.

Consider a differential equation
\begin{align}
L\bullet f(x) = 0,\qquad
L = p_r(x) \p_x^r +p_{r-1}(x) \p_x^{r-1} + \dots + p_1(x)\p_x + p_0(x),
\end{align}
where $p_0(x),\dots,p_r(x)$ are coprime polynomials.
Let $d$ be the degree of $p_r(x)$.
We assume that $\deg p_k(x) \leq d$ $(k=0,1,\dots,r-1)$.

By the definition of a regular singular point, if $x=0$ is a regular singular point, then $d\leq r$.
Similarly, if $x=\infty$ is a regular singular point, then $d\geq r$.

If $x=0$ is a regular singular point, the main terms of the differential equation in the neighborhood of $x=0$ are $x^d\p_x^r,\ x^{d-1}\p_x^{r-1},\ \dots,\ \p_x^{r-d}$.
Because of the relation $x^a\p_x^b = x^{a-b}\theta_x(\theta_x-1)\cdots(\theta_x-b+1)$,
the indicial equation has $0,\ 1,\ \dots,\ r-d-1$ as its roots.

The regular singular point and its exponents are transformed by the Fourier transform as follows.

\begin{proposition}
Suppose that $p_r(x)=x^d$, $\deg p_k(x) \leq d$ $(k=0,1,\dots,r-1)$, $x=0$ is a regular singular point for the equation $L\bullet f(x)=0$ and its exponents are $\mu_1,\dots,\mu_d$ and $0,\ 1,\ \dots,\ r-d-1$. Then
$x=\infty$ is a regular singular point for the Fourier transformed equation $\hat L \bullet \hat f(x) = 0$ and its exponents are $-\mu_1-1,\ -\mu_2-1,\ \dots,\ -\mu_d-1$.

Suppose that $\deg p_k(x) \leq \deg p_r(x)$ $(k=0,1,\dots,r-1)$, $x=\infty$ is a regular singular point for the equation $L\bullet f(x)=0$ and its exponents are $\mu_1,\dots,\mu_r$.
Then $x=0$ is a regular singular point for the transformed equation and its exponents are $-\mu_1-1,\ -\mu_2-1,\ \dots,\ -\mu_r-1$ and $0,\ 1,\ \dots,\ d-r-1$.
\end{proposition}
\begin{proof}
By the assumption there exists a term with the highest degree both in $x$ and $\p_x$.
By the Fourier transform $x\mapsto i\p_x,\ \p_x \mapsto ix$, the highest degree is not changed and 
their weights are reversed.
Hence if $x=0\ (\infty)$ is a regular singular point for $L\bullet f=0$, then $x=\infty\ (0)$ is a regular singular point for $\hat L\bullet \hat f=0$, and the main terms in the neighborhood of the singular point are not changed by the Fourier transform.

If $x=0$ is a regular singular point for $L\bullet f(x)=0$, its main term equals $x^{-(r-d)}\theta_x(\theta_x-1)\cdots(\theta_x-(r-d)+1)(\theta_x-\mu_1)\cdots(\theta_x-\mu_d)$.
The main term is Fourier transformed to
\begin{align}
&\quad \p_x^{-(r-d)}(\theta_x+1)(\theta_x+2)\cdots(\theta_x+(r-d))(\theta_x+\mu_1+1)\cdots(\theta_x+\mu_d+1)\\
&= x^{r-d}(\theta_x+\mu_1+1)\cdots(\theta_x+\mu_d+1),
\end{align}
by the formula $\p_x^k x^k = (\theta_x+1)(\theta_x+2)\cdots(\theta_x+k)$.
This gives the exponents for the Fourier transformed equation at $x=\infty$.

If $x=\infty$ is a regular singular point for $L\bullet f(x)=0$,
its main term equals $x^{d-r}(\theta_x-\mu_1)\cdots(\theta_x-\mu_r)$.
The main term is Fourier transformed to
\begin{align}
&\quad \p_x^{d-r}(\theta_x+\mu_1+1)\cdots(\theta_x+\mu_r+1)\\
&= x^{-(d-r)}\theta_x(\theta_x-1)\cdots(\theta_x-(d-r)+1)(\theta_x+\mu_1+1)\cdots(\theta_x+\mu_r+1),
\end{align}
by the formula $x^k\p_x^k = \theta_x(\theta_x-1)\cdots(\theta_x-k+1)$.
This gives the exponents for the Fourier transformed equation at $x=0$.
\end{proof}

\section{Applications to statistics}
\label{sec:applications}

\subsection{Sum of beta random variables}
\label{subsec:beta}

Let $f_n(a,b;x)$ be the probability density function of sum of $n$ beta random variables $\mathrm{Beta}(a,b)$.
The moment generating function of the beta random variable is
\begin{align}
M(a,b;t)
= \frac{\Gamma(a+b)}{\Gamma(a)\Gamma(b)} \int_0^1 e^{tu} u^{a-1}(1-u)^{b-1}du.
\end{align}
Since $M(a,b;t)$ equals the confluent hypergeometric function ${}_1F_1(a,a+b;t)$ (c.f. \cite{NIST_book_10}), the characteristic function $\phi(a,b;t) = M(a,b;it)$ satisfies the following second order differential equation:
\begin{align}
\bigg[ \p_t^2 - \Big(i-\frac{a+b}{t} \Big)\p_t - \frac{ia}{t} \bigg]
\bullet \phi(a,b;t) = 0,\qquad
(i=\sqrt{-1}).
\label{ode beta}
\end{align}

An $(n+1)$-th order differential equation satisfied by $M(a,b;t)^n$ is derived by Theorem \ref{thm:ode}, and by the Fourier transform, we obtain a holonomic differential equation satisfied by $f_n(a,b;x)$.
By Corollary \ref{cor:con}, putting $m_0=M_0=-1,\ m_1=-1,\ M_1=0$, the equation for $f_n(x)$ is at most of the $n$-th order.
In fact, the equation derived by the procedure of Section 2 is exactly of the $n$-th order.

We define the initial term of a formal power series, before the proposition on the order.
\begin{definition}
For a formal power series $f(x)= \sum_{k=0}^\infty c_n x^{\lambda + k}$ $(c_0\neq 0)$, we define 
\begin{align}
\ini f(x) = c_0 x^\lambda.
\end{align}
We denote the matrix (vector) whose $(i,j)$ entry is $\ini f_{ij}(x)$ by $\ini F(x)$, where $F(x)=(f_{ij}(x))$.
\end{definition}

\begin{proposition}
\label{prop:beta}
The differential equation for $f_n(a,b;x)$ derived by the procedure of Section 2 is of the $n$-th order.
\end{proposition}
\begin{proof}
We prove
\begin{align}
\ini q_{0,j}(t) = i(-1)^j   n\frac{a}{b}(a+b)_{j-1} t^{-(j-1)},\quad
\ini q_{1,j}(t) = (-1)^{j-1} n (a+b)_{j-1}t^{-(j-1)},
\end{align}
for $j=2,3,\dots,n+1$ by induction based on \eqref{eq:qij}.
It is easy to check them for $j=2$.
Assuming them up to $j$, we have
\begin{align}
&\ini q_{0,j+1}(t)
= \p_t q_{0,j}(t) + a_0(t) \ini q_{1,j}(t)
= i(-1)^{j+1} n \frac{a}{b} (a+b)_{j}t^{-j},\\
&\ini q_{1,j+1}(t)
= (\p_t - bt^{-1}) \ini q_{1,j}(t)
= (-1)^{j} n (a+b)_j t^{-j},
\end{align}
since $\mindeg q_{2,j} \geq -(j-2)$ by \eqref{degin}.
Thus, the results are shown by induction.

Hence, we obtain $\mindeg q_{0,n+1}(t) = -n$.
As in the proof of Theorem \ref{thm:deg}, by clearing the denominators, 
we see that the highest degree of $t$ of the equation \eqref{ode n} for $\phi(a,b;t)^n$ is $n$.
\end{proof}

\begin{example}
\label{ex:f_3}
$f_3(a,b;x)$ satisfies the differential equation
\begin{align}
&\Big[
x(x-1)(x-2)(x-3)\p_x^3 + \big(-6(a+b-2)x^3 + 2(16a + 11b - 27)x^2 -6(8a + 3b - 11)x\\
& + 18(a-1)\big)\p_x^2 + \big((a+b-2)(11(a+b)-18)x^2 - (48a^2 + 66ab + 18b^2 - 145a - 95b\\
& + 108)x + 3(a - 1)(15a + 12b - 22)\big)\p_x - (a+b-2)(2(a+b)-3)(3(a+b)-4)x\\
& + 3(a-1)(2(a+b)-3)(3(a+b)-4)
\Big] \bullet f_3(a,b;x) = 0.
\end{align}
\end{example}

Note that $x=0$ is a regular singular point for \eqref{ode beta}, and its exponents are $0,1-(a+b)$ since
\begin{align}
\p_t^2 - \Big(i-\frac{a+b}{t} \Big)\p_t - \frac{ia}{t}
= t^{-2} \theta_t(\theta_t -1+a+b) -it^{-1}(\theta_t + a).
\end{align}
Hence we can obtain the exponents for the equation satisfied by $f_n(a,b;x)$ relative to regular singular point $\infty$.
However, it is not informative since $f_n(a,b;x)$ has a compact support.

On the other hand, the equation in Example \ref{ex:f_3} has regular singular points at $x=0,1,2,3$.
In general, the degree of the coefficient polynomial of the highest order term $\p_x^n$ is less than or equal to $n+1$.
On the other hand, a differential equation satisfied by $f_n$, or the $n$-th convolution of $f_1$, has to have singular points at $x=0,1,\dots,n$, because $f_1$ has singular points at $x=0,1$.
Therefore, the highest order term of the differential equation derived as mentioned above is $x(x-1)\cdots(x-n)\p_x^n$,
and this implies that  $x=0,1,\dots,n$ are all regular singular points.

Especially, in the case of $a=b=1$ (then, the beta distribution becomes the uniform distribution), the differential equation is simply $x(x-1)\cdots(x-n) \p_x^n \bullet f_n(1,1;x)=0$.
This is because $\phi(1,1;t) = (e^{it}-1)/(it)$ and thus
\begin{align}
&\quad 2\pi \mathcal F^{-1}[x(x-1)\cdots(x-n) \p_x^n \bullet f_n(1,1;x)]\\
&= -i\p_t(-i\p_t-1)\cdots(-i\p_t-n) \bullet (-1)^n (e^{it}-1)^n\\
&= -i\p_t(-i\p_t-1)\cdots(-i\p_t-(n-1)) \bullet (-1)^n n(e^{it}-1)^{n-1}\\
&=\cdots =  -i\p_t \bullet (-1)^n n! = 0.
\end{align}
This shows that $f_n(1,1;x)$ has to be a piece-wise $(n-1)$-th degree polynomial.

The exact form of $f_n(1,1;x)$ is given in Section 1.9 of \cite{feller_vol2_book}.
For $n=2$, $f_2(1,1;x)$ is a continuous piece-wise linear function and $f_2\in C^0$.
By induction it follows that $f_n(1,1;x)=\int_0^1 f_{n-1}(1,1;x-y)dy$ belongs to $C^{n-2}$.
Hence we can put $f_n(1,1;x) = c_0x^{n-1}\ (0\leq x \leq1)$ by the smoothness at $x=0$.
We can also put $f_n(1,1;x) = c_0x^{n-1} + c_1(x-1)^{n-1}\ (1\leq x \leq2)$ by the smoothness at $x=1$.
In the same way, we can put $f_n(1,1;x)= \sum_{j=0}^{k} c_j (x-j)^{n-1}\ (k\leq x \leq k+1)$.
By the smoothness at $x=n$, the $k$-th $(k=0,1,\dots,n-2)$ derivative of $g_n(x) = \sum_{j=0}^{n-1} c_j (x-j)^{n-1}$ at $x=n$ is zero.
Moreover, we have
\begin{align}
\sum_{j=0}^{n-1} \int_j^n c_j (x-j)^{n-1} dx
= \sum_{j=0}^{n-1} c_j\frac{(n-j)^n}{n}
= 1,
\end{align}
because of $\int_0^n f_n(1,1;x) dx = 1$.
Hence, $c_0,\dots,c_{n-1}$ satisfy the equation
\begin{align}
\bma
n^{n} & (n-1)^{n} &\cdots & 1^{n}\\ 
n^{n-1} & (n-1)^{n-1} &\cdots & 1^{n-1}\\ 
\vdots & \vdots & \vdots & \vdots \\
n^1 & (n-1)^1 &\cdots & 1^1
\ema
\bma
c_0\\
c_1\\
\vdots\\
c_{n-1}
\ema
=
\bma n \\ 0\\ \vdots \\0 \ema.
\end{align}
The matrix on the left is invertible by the Vandermonde determinant.
Therefore, we can determine the probability density function $f_n(1,1;x)$.

\begin{remark}
The $k$-th moment of the beta random variable is $(\alpha)_k/(\alpha+\beta)_k$.
The moments of sum of $n$ beta random variables are the coefficients of
\begin{align}
\bigg(
\frac{(\alpha)_0}{(\alpha+\beta)_0} + \frac{(\alpha)_1}{(\alpha+\beta)_1}t^1 + \frac{(\alpha)_2}{(\alpha+\beta)_2}t^2 + \cdots
\bigg)^n.
\end{align}
The probability density function can be approximated in terms of orthogonal polynomials by fitting the moments.  However the information provided by the differential equation can not be easily derived from the
moments.
\end{remark}

\subsection{Sum of cubes of standard normal random variables}
\label{subsec:cube}
In this section we study characteristic functions and probability density functions of sum
of cubes of standard normal variables.    Concerning the probability distribution of sample skewness from
normal population, Geary (\cite{geary}) and Mulholland (\cite{mulholland}) give very detailed results.
However the distribution of the sum of cubes of standard normal variables, which is a more basic quantity
than the sample skewness, has not been studied in detail.

Let $f_n(x)$ denote  the probability density function of sum of cubes of $n$ standard normal variables.
The characteristic function of the cube of a standard normal variable is
\begin{align}
\phi(t)
= \frac{1}{\sqrt{2\pi}}\int_{-\infty}^\infty e^{-x^2/2} e^{itx^3} \,dx,\qquad
(i=\sqrt{-1}).
\end{align}
Let
\begin{align}
I_j(t)
= \frac{i^j}{\sqrt{2\pi}}\int_{-\infty}^\infty x^j e^{-x^2/2} e^{itx^3}\,dx,
\end{align}
for $j=0,1,2,\dots$.
$I_j(t)$ satisfies the recursion
\begin{align}
3t I_j(t)
&= \frac{i^{j-1}}{\sqrt{2\pi}}\int_{-\infty}^\infty x^{j-2} e^{-x^2/2} \, \Big(e^{itx^3}\Big)' \,dx
= -\frac{i^{j-1}}{\sqrt{2\pi}}\int_{-\infty}^\infty e^{itx^3} \,\Big(x^{j-2}e^{-x^2/2}\Big)' \,dx \\
&= I_{j-1}(t) + (j-2)I_{j-3}(t).
\end{align}
Since $\phi(t)=I_0(t),\ \p_t\phi(t) = -I_3(t),\ \p_t^2\phi(t) = I_6(t)$, we can express 
$\p_t\phi(t), \p_t^2\phi(t)$ in terms of $I_0(t), I_1(t)$.
By eliminating $I_0(t), I_1(t)$, we obtain the second-order differential equation for $\phi(t)$
\begin{align}
\Big[ 27t^3\p_t^2 + (81t^2+1)\p_t + 15t \Big]\bullet \phi(t)=0. \label{phi 1}
\end{align}
This can be also derived by integration algorithm (\cite{oaku1997}).  

A differential equation satisfied by $\phi(t)^n$ is derived by Theorem \ref{thm:ode}, and by Fourier transform, we obtain a differential equation satisfied by $f_n(x)$.
In Corollary \ref{cor:con}, putting $m_0=M_0=-2,\ m_1=-3,\ M_1=-1$, the equation for $f_n(x)$ is at most of the $3n$-th order.
In fact, it is exactly of the $3n$-the order.

\begin{proposition}
The differential equation for $f_n(x)$ derived by the procedure of Section 2 is of the $3n$-th order.
\end{proposition}
\begin{proof}
As in the proof of Proposition \ref{prop:beta}, we can show
\begin{align}
\ini q_{0,j}(t) = (-1)^{j-1} 5n(3t)^{-3(j-1)+1},\quad
\ini q_{1,j}(t) = (-1)^{j-1}  n(3t)^{-3(j-1)},
\end{align}
for $j=2,3,\dots,n+1$.
Hence, we obtain $\mindeg q_{0,n+1}(t) = -3n$,
and thus the highest degree of $t$ of the equation \eqref{ode n} is $3n$.
\end{proof}

\begin{example}
$f_4(x)$ satisfies the differential equation
\begin{align}
&\Big[
177147 x^5\p_x^{12} +5314410 x^4\p_x^{11} +52455195 x^3\p_x^{10} +(65610 x^4+202242825 x^2)\p_x^9\\
&+ (1180980 x^3+278372295 x)\p_x^8+ (6145470 x^2+89579520)\p_x^7+ (8505 x^3+9950850 x)\p_x^6\\
&+ (76545 x^2+3408480)\p_x^5+155655 x\p_x^4 +(450 x^2+56160)\p_x^3 +1350 x\p_x^2 +480 \p_x+8 x
\Big]\\
& \bullet f_4(x)
=0.
\end{align}
\end{example}

Note that $x=\infty$ is a regular singular point of the differential equation \eqref{phi 1}, and its exponents are $-1/3, -5/3$
since
\begin{align}
27x^3\p_x^2 + (81x^2+1)\p_x + 15x
= 3x(3\theta_x+1)(3\theta_x+5) + x^{-1}\theta_x.
\end{align}
This implies  that there exists a differential equation satisfied by $\phi(t)^n$ which is regular at $x=\infty$ and its exponents are $-n/3,\ -(n+4)/3,\ \dots,\ -(5n-4)/3,\ -5n/3$.
Moreover, there exists a differential equation satisfied by $f_n(x)$ which is regular at $x=0$ and its exponents are $n/3-1,\ (n+4)/3-1,\ \dots,\ (5n-4)/3-1,\ 5n/3-1$ and $0,\ 1,\ \dots,\ 2n-2$.

We now briefly discuss issues in numerical evaluation of $f_n$ based on our differential equation
and computation of initial values.
For numerically solving the differential equation satisfied by $f_n$, an initial value of $(f_n,\ \p_x \bullet f_n, \dots, \p_x^{3n-1} \bullet f_n)$ at $x = x_0 \neq 0$ is needed.  Note that we can not use  $x=0$ 
as the initial point, because it is the singular point of the differential equation.

By dividing the interval of integration of the inversion formula for the characteristic function 
and integrating by parts repeatedly, for any integer $m$ $(\geq 0)$, we obtain
\[
f_n(x)
= \frac{1}{\pi} \Re \bigg[ \int_{0}^T  \phi_n(t) e^{itx} \,dt 
+ e^{iTx}  \sum_{j=1}^{m} \phi_n^{(j-1)}(T) \bigg(\frac{i}{x}\bigg)^j 
+ \bigg(\frac{i}{x}\bigg)^m \int_T^\infty \phi_n^{(m)}(t) e^{itx}\,dt \bigg],
\]
where $\phi_n(t) = \phi(t)^n$.  Integration by parts is needed for numerical evaluation of derivatives of $f_n$.

The formal $k$-th derivative of $f_n(x)$ at $x=x_0$ is
\begin{align}
f_n^{(k)}(x_0)
&= \frac{1}{\pi} \Re \bigg[ \int_{0}^T  (it)^k \phi_n(t) e^{ix_0t} \,dt
+ e^{iTx_0} \sum_{j=1}^{m} \phi_n^{(j-1)}(T) \sum_{l=0}^k \binom{k}{l} \frac{(j+l-1)!}{(j-1)!} \frac{i^{k+j+l}}{x_0^{j+l}}T^{k-l}\notag\\
&\qquad + \sum_{l=0}^k \binom{k}{l} \frac{(m+l-1)!}{(m-1)!} \frac{i^{m+k+l}}{x_0^{m+l}} \int_T^\infty t^{k-l} \phi_n^{(m)}(t) e^{ix_0t}\,dt
 \bigg]. \label{k-th}
\end{align}

Define
\begin{align}
n!m = \prod_{0\leq k < n/m} (n-km),\ (n>0),\qquad
n!m = 1,\ (n\leq 0).
\end{align}
\begin{lemma}\label{lem:exp}
The following expansion of $\phi(t)$ holds at any $t\ (\neq0)$:
\begin{align*}
\phi(t)
&= \frac{\sqrt{2\pi}}{3\Gamma(2/3)}|t|^{-1/3}
\sum_{k=0}^\infty \frac{(6k-5)!6}{(6k)!6\,(6k-4)!6}(3t)^{-2k}
 -\frac{\sqrt{2\pi}}{9\Gamma(1/3)}|t|^{-5/3}
\sum_{k=0}^\infty \frac{(6k-1)!6}{(6k)!6\,(6k+4)!6} (3t)^{-2k}.
\end{align*}
\end{lemma}

\begin{proof}
Airy function $\rm{Ai}(x)$ can be written as follows (\cite{copson}, \cite{vallee}, Chapter 9 of \cite{NIST_book_10}):
\[
\mathrm{Ai}(x)
= \frac{1}{\pi} e^{-\frac{2}{3}x^{3/2}} \int_0^\infty e^{-\sqrt{x}u^2} \cos\big( u^3/3\big) \,du \qquad(x>0). 
\]
By the transform $u=2^{-1/2}x^{-1/4}s$, we obtain
\begin{align}
\mathrm{Ai}(x)
&= \frac{1}{\sqrt{2}\pi}x^{-1/4} e^{-\frac{2}{3}x^{3/2}} \int_0^\infty e^{-s^2/2} \cos\bigg( \frac{x^{-3/4}}{2^{3/2}} \frac{s^3}{3} \bigg) \,ds.
\end{align}
From this and
\begin{align}
\phi(t) = \sqrt{\frac{2}{\pi}}\int_0^\infty e^{-x^2/2} \cos(tx^3)\, dx,
\end{align}
we obtain the following relation between $\phi(t)$ and $\rm{Ai}(t)$:
\begin{align}
\phi(t)
&= \frac{\sqrt{2\pi}}{3^{1/3}} t^{-1/3} e^{t^{-2}/108}  \mathrm{Ai}\bigg( \frac{t^{-4/3}}{4\cdot 3^{4/3}}\bigg)
\qquad(t>0). \label{ai}
\end{align}
Since the Maclaurin expansion of $\rm{Ai}(t)$ is known, the expansion of $\phi(t)$ in $1/t$ can be derived.
\end{proof}

By \eqref{ai} we see that $\phi(t) = O(t^{-1/3})\ (t\to\infty)$.
Hence, \eqref{k-th} is justified for $k<m+n/3$ since the integral on $[T,\infty)$ converges
for $k<m+n/3$.

In \eqref{k-th} the  integral on $[0,T]$ can be numerically computed by using the relation \eqref{ai}.

The integral on $[T,\infty)$ can be computed without numerical integral.
The values of $J_l = \int_{T'}^\infty e^{it}/t^{l/3} \, dt \ (l=1,2,\dots)$, where $T' = Tx_0$, is needed for computation of the integral on $[T,\infty)$.
$J_l$ satisfies the following recurrence relation:
\begin{align}
J_{l+3}
= \frac{3}{l} \bigg( \frac{e^{iT'}}{{T'}^{l/3}} + i J_l \bigg),
\end{align}
and thus it is sufficient to compute $J_1$, $J_2$ and $J_3$.
$J_1$ and $J_2$ can be reduced to the integral on $[0,T]$ from the formulae (cf.\ Section 5.9 of \cite{NIST_book_10})
\begin{align}
\int_0^\infty \frac{\sin x}{x^p}\,dx = \frac{\pi}{2\Gamma(p)\sin(p\pi/2)},\qquad
\int_0^\infty \frac{\cos x}{x^p}\,dx = \frac{\pi}{2\Gamma(p)\cos(p\pi/2)}\qquad(0<p<1).
\end{align}
$J_3$ can be computed  by Maclaurin expansion of trigonometric integrals.

The second term of the right side of \eqref{k-th} can be computed by Lemma \ref{lem:exp}.

From the above, the value of $f_n^{(k)}(x_0)\ (n=0,1,\dots,3n-1)$ can be computed and the differential equation satisfied by $f_n(x)$ can be numerically solved.

\section{Some discussions}
\label{sec:discussion}

In this paper we investigated properties of powers of functions 
satisfying a second-order holonomic differential equation.  
Our motivating example was the distribution of convolutions of cubes of standard normal random variables presented in Section \ref{subsec:cube}, 
which was in turn motivated by the algorithm given in \cite{Oaku2013}. 
In the course of our study of distribution of cubes of standard normal random variables, 
we noticed some remarkable properties satisfied by the characteristic function of
the cube of a standard normal random variable.
Based on this example, we developed more general theory presented in Section \ref{sec:main},
which may be relevant to problems in other areas of applied mathematics.

From a mathematical viewpoint, it is of interest to generalize the results of Section \ref{sec:main} to the case of powers of a general holonomic function. 
From a statistical viewpoint, it is of interest to investigate the distribution of the sum of 
the $r$-th power ($r\ge 4$) of standard normal random variables.

\section*{Acknowledgment}

We are grateful to C. Koutschan for computation of \eqref{eq:2f1-recursion}.




\bibliographystyle{abbrv}
\bibliography{second-order-holonomic}

\end{document}